\documentclass[11pt]{article}
\usepackage{amsmath, amsthm, amssymb,amscd, amsfonts}

\usepackage{tkz-euclide, float}
\usetikzlibrary{shapes, calc} 
\usetikzlibrary{bbox}

\usepackage{enumerate, enumitem, bbm, bm}

\usepackage{fullpage, verbatim, subcaption, array, tabularx, comment}
\usepackage{tikz, xcolor, graphicx}
\tikzstyle{uStyle}=[shape = circle, minimum size = 20pt, inner sep =2.5pt, outer sep = 0pt, draw, fill=white]
\tikzstyle{myStyle}=[shape = circle, draw, fill=black, scale=0.5]

\usepackage{url,hyperref}
\hypersetup{colorlinks=true,linkcolor=black,anchorcolor=blue,citecolor=black,filecolor=blue,urlcolor=blue,bookmarksnumbered=true,pdfview=FitB}

\usepackage{thmtools}
\usepackage{thm-restate}
\usepackage{soul}
\usepackage{cleveref}
\usepackage[all,arc]{xy}
\usepackage{mathrsfs}
\usepackage{mathtools}

\usepackage{todonotes}

\RequirePackage{marginnote,hyperref}
\newcommand{\aside}[1]{\marginnote{\scriptsize{#1}}[0cm]}
\newcommand{\aaside}[2]{\marginnote{\scriptsize{#1}}[#2]}
\newcommand\Emph[1]{\emph{#1}\aside{#1}}
\newcommand\EmphE[2]{\emph{#1}\aaside{#1}{#2}}

\newtheorem{lem}{Lemma}
\newtheorem{cor}[lem]{Corollary}
\newtheorem{claim}{Claim}
\newtheorem*{mainthm}{Main Theorem}
\newtheorem*{key-lem}{Key Lemma}

\newcommand{\vph}{\varphi}
\def\diam{\textrm{diam}}
\def\dist{\textrm{dist}}
\def\ch{\textrm{ch}}
\newcommand{\vw}{v^*w}
\newcommand{\Vz}{\widehat{V}}

\title{5-Coloring Reconfiguration of Planar Graphs \\ with No Short Odd Cycles}
\author{Daniel W. Cranston\thanks{Department of Computer Science, Virginia Commonwealth University, Richmond, VA, USA; \texttt{dcranston@vcu.edu}} \and Reem Mahmoud\thanks{Department of Computer Science, Virginia Commonwealth University, Richmond, VA, USA; \texttt{mahmoudr@vcu.edu}}}

\date{\today}
\newcommand\C{\mathcal{C}}

\begin{document}

\maketitle

\begin{abstract}
The coloring reconfiguration graph $\mathcal{C}_k(G)$ has as its vertex set all
the proper $k$-colorings of $G$, and two vertices in $\mathcal{C}_k(G)$ are
adjacent if their corresponding $k$-colorings differ on a single vertex.
Cereceda conjectured that if an $n$-vertex graph $G$ is $d$-degenerate and
$k\geq d+2$, then the diameter of $\mathcal{C}_k(G)$ is $O(n^2)$. Bousquet and
Heinrich proved 
that if $G$ is planar and bipartite, then 
the diameter of $\mathcal{C}_5(G)$ is $O(n^2)$. 
(This proves Cereceda's Conjecture for every such graph with degeneracy 3.)
They also highlighted the particular case of
Cereceda's Conjecture when $G$ is planar and has no 3-cycles.
As a partial solution to this problem, we show that the diameter of
$\mathcal{C}_5(G)$ is $O(n^2)$ for every
planar graph $G$ with no 3-cycles and no 5-cycles.
\end{abstract}

\section{Introduction}

Given two $k$-colorings, $\vph_1$ and $\vph_2$, of a graph $G$, can we
transform $\vph_1$ to $\vph_2$ by a sequence of steps each of which recolors a
single vertex?  More precisely, we require that after each recoloring step
the resulting intermediate coloring is again a proper $k$-coloring.  It is
convenient to study the graph $\C_k(G)$ which has
as its vertices the set of all proper $k$-colorings of $G$, with vertices of
$\C_k(G)$ adjacent when their colorings differ on a single vertex of $G$.  Our
original question is answered yes precisely when $\vph_1$ and $\vph_2$ lie in
the same component of $\C_k(G)$.  Similarly, we can ask whether $\C_k(G)$ is
connected and, if it is, ask for bounds on its diameter in terms of $|G|$.

This area of study is known as \emph{reconfiguration}.  It is natural to ask
similar questions for any problem where we have a set of solutions (above,
these were proper $k$-colorings) and a well-defined reconfiguration step
(above, recoloring a single vertex).  Examples of topics studied through this
lens include perfect matchings, spanning trees, independent sets, dominating
sets, and solutions to an instance of 3-SAT, among many others.
For a thorough introduction to this
topic, we refer the reader to  surveys by Nishimura~\cite{Nishimura} and van
den Heuvel~\cite{vandenHeuvel}.

Throughout this paper, we use standard graph theory terminology and notation,
as in \cite{West}. Let $n:=|V(G)|$ and let $dist(v,w)$ denote the length of the
shortest $v,w$-path. The diameter of $G$, denoted \Emph{diam$(G)$}, is
$\max_{v,w\in V(G)}\dist(v,w)$. A proper $k$-coloring of a graph $G$ is a map
$\vph:V(G)\to[k]$, where $[k]:=\{1,\dots,k\}$, such that $\vph(v)\neq\vph(w)$
whenever $vw\in E(G)$. The \emph{reconfiguration graph} \Emph{$\mathcal{C}_k(G)$} has
as its vertex set all proper $k$-colorings of $G$, and two vertices in
$\mathcal{C}_k(G)$ are adjacent if their corresponding $k$-colorings differ on
a single vertex of $G$. 

Cereceda, van den Heuvel, and Johnson \cite{CvJ, CvJ2} were the first to study
the connectedness of $\mathcal{C}_k(G)$. Subsequently, many
results~\cite{BBFHMP, BousquetHeinrich, DvorakFeghali, Carl}
have focused on the diameter of $\mathcal{C}_k(G)$; in particular, these
papers attempt to tackle Cereceda's
Conjecture that diam$(\mathcal{C}_k(G))=O(n^2)$ whenever $G$ is
$d$-degenerate and $k\geq d+2$. 
(The bound $O(n^2)$ is best possible, as shown by Bonamy et al.~\cite{BJLPP})
Most results on the connectedness of
$\mathcal{C}_k(G)$ give an exponential upper bound on its diameter. Bartier et
al. \cite{BBFHMP} proved that $\diam(\mathcal{C}_5(G))=O(n)$ for every
planar graph $G$ of girth at least 6, while Dvo\v{r}\'{a}k and Feghali
\cite{DvorakFeghali} proved $\diam(\mathcal{C}_7(G))=O(n)$ for every
triangle-free planar graph $G$. Feghali~\cite{Carl} showed that if $d\geq1$ and
$k\geq d+1$, then for every $\epsilon>0$ and every graph $G$ with maximum
average degree $d-\epsilon$, we have $\diam(\mathcal{C}_k(G))=O(n(\log
n)^{d-1})$. 

In a recent breakthrough, Bousquet and Heinrich
\cite{BousquetHeinrich} proved that $\diam(\mathcal{C}_k(G))=O(n^{d+1})$
for every $d$-degenerate graph $G$ with $k\geq d+2$. Since planar graphs are
5-degenerate, and triangle-free planar graphs are 3-degenerate, their result
implies that $\diam(\mathcal{C}_7(G))=O(n^6)$ for every planar graph
$G$ and $\diam(\mathcal{C}_5(G))=O(n^4)$ for every
triangle-free planar graph $G$. In the same paper, they proved that
$\diam(\mathcal{C}_5(G))=O(n^2)$ for every bipartite planar graph $G$.
They also remarked that Cereceda's Conjecture remains open for triangle-free
planar graphs. Our Main Theorem is a step towards proving Cereceda's Conjecture
for all triangle-free planar graphs.

\begin{mainthm}
If $G$ is a 
planar graph with no 3-cycles and no 5-cycles, then
$\diam(\mathcal{C}_5(G))\le 4n^2$. 
\end{mainthm}

\section{Proof of Main Theorem}

Before starting the proof, we review some relevant (mostly standard)
definitions.

Let $G$ be a plane graph.  Denote by $F(G)$ the set of faces of $G$
and by $d(f)$ be the length of each such face $f$. A
\emph{$k$-vertex}\aaside{\mbox{$k$-/$k^+$-/$k^-$-} vertex/ neighbor/
face}{-7.5mm} is one with degree $k$. A \emph{$k^+$-vertex} (resp.~%
\emph{$k^-$-vertex}) is one with degree at least (resp.~at most) $k$.
A \emph{$k$-neighbor} 
is an adjacent $k$-vertex. A \emph{$k$-face} 
is one of length $k$. We define \emph{$k^+$-neighbor, $k^-$-neighbor,
$k^+$-face}, and \emph{$k^-$-face} analogously. A cycle $C$ is
\EmphE{separating}{0mm} if $G\setminus V(C)$ is disconnected. A vertex
$v\in V(G)$ is \EmphE{inner}{0mm} if $v$ does not lie on the outer face.
A plane graph $G$ is \emph{Type 1} if $\delta(G)\geq3$ and \emph{Type 2}\aside{Type 1, 2} if $\delta(G)\geq2$, the outer face $f_0$ is a 7-face, $V(f_0)\subsetneq V(G)$, and every 2-vertex of $G$ lies on $f_0$. 
(Note that Type 1 and Type 2 graphs are not necessarily mutually exclusive.
That is, a graph $G$ could satisfy the criteria to be both Type 1 and Type 2.
In that case, we will always define $G$ to be Type 2, unless we explicitly say
otherwise, since our conclusions for Type 2 graphs will be stronger.)
Let $T:=\emptyset$ if $G$ is Type 1 and $T:=V(f_0)$ if $G$ is Type 2. Let $V_1:=\{v\in V(G)\setminus T: d_G(v)\leq3\}$, $V_2:=\{v\in V(G)\setminus T: d_{G\setminus V_1}(v)\leq3\}$, $V_3:=\{v\in V(G)\setminus T: v\notin V_1\cup V_2\}$, and $V_4:=\{v\in T\}$. For all
$i\geq1$, each vertex $v\in V_i$ has \EmphE{level}{0mm} $i$.

Let $v$ be a 3-vertex with neighbors $v_1,v_2,v_3$, and assume that $v$ is
incident with 3 distinct faces.  For each $i\in[3]$, the
\EmphE{face opposite}{0mm} $v_i$ (with respect to $v$) is the face incident
with vertex $v$ that is not incident
with edge $vv_i$. A 3-vertex $v$ is \emph{good} \aaside{good
(neighbor)}{0mm} if it has a neighbor $w$ of level at most 2 whose opposite
face is a 4-face. We call $w$ a \emph{good neighbor} of $v$. 

Bousquet and Heinrich \cite{BousquetHeinrich}, in their proof for bipartite
planar graphs,  introduced the notion of the \emph{level of a vertex}. They
used it to
identify a pair of vertices in $G$ at distance 2 along a common face that could
be identified to proceed by induction.
We will apply a similar technique for planar graphs with no
3-cycles and no 5-cycles. However, identifying such vertices in these graphs
might create a 5-cycle, if the vertices lie on a 7-cycle (it will not create any 3-cycle
since $G$ has no 5-cycle).

To avoid this problem, we
show how to find good vertices inside a subgraph of our graph which does
not contain any separating 7-cycles; note that identifying a pair of vertices
in such a subgraph cannot create a 5-cycle. In particular, if our graph has
a separating 7-cycle, then we pick an innermost such cycle and, using
discharging, we show that we can use the extra charge from Euler's formula to
ensure we find a good vertex away from the outer face.  This technique has been
used previously; for example, see~\cite{DKT, DP}.
Our next lemma and subsequent corollary establish the
existence of good vertices inside separating 7-cycles.

\begin{key-lem}
\hypertarget{key-lemma}{}
Let $G$ be a connected plane graph with no 3-cycles, no 5-cycles, and no separating
7-cycles. If $G$ is Type 1, or $G$ is Type 2, then $G$ contains a good vertex $v$ with a good neighbor $w$ such that $v,w\notin T$.
\end{key-lem}

\begin{proof}
Let $\Vz:=V(f_0)$. The proofs for Type 1 graphs and Type 2 graphs are similar,
though for Type 2 graphs the proof is harder. Recall that if $G$ is Type 1,
then $T:=\emptyset$ and if $G$ is Type 2, then $T:=\Vz$. We use discharging
to show that $V(G)\setminus T$ contains some good vertex and its good neighbor. 
Assume instead that $G$ is a counterexample to the lemma. Denote by
\emph{$\ch(v)$} and \emph{$\ch(f)$} (resp.~\emph{$\ch^*(v)$} and
\emph{$\ch^*(f)$}) the initial (resp.~final) charges of each vertex $v$ and each face
$f$. Let $\ch(v):=d(v)-4$ for every $v\in V(G)$ and $\ch(f):=d(f)-4$ for every
$f\in F(G)$. Using Euler's formula, the total initial charge is $-8$:
\begin{align}
\sum_{v\in V(G)}(d(v)-4)+\sum_{f\in F(G)}(d(f)-4) &=\sum_{v\in
V(G)}d(v)-4|V(G)|+\sum_{f\in F(G)}d(f)-4|F(G)| \nonumber \\
&=2|E(G)|-4|V(G)|+2|E(G)|-4|F(G)| \nonumber \\ &=4(|E(G)|-|V(G)|-|F(G)|) \nonumber \\ &=-8.
\label{charge-sum}
\end{align}

Now we redistribute charge according to rules (R1) and (R2) below and show that
the total final charge is greater than $-8$, a contradiction.

\begin{itemize}
    \item[(R1):] Every $6^+$-face $f$ gives $\frac{d(f)-4}{d(f)}$ to every
incident vertex.\footnote{If a vertex $v$ appears multiple times on a boundary
walk of $f$, then $f$ gives $v$ charge $\frac{d(f)-4}{d(f)}$ for each such incidence.}
    \item[(R2):] Every 3-vertex not in $T$ takes $\frac{1}{3}$ from every
neighbor that has level at least 3.
\end{itemize}

Assume first that $G$ is Type 1. We show that 
$\ch^*(v)\ge 0$ and $\ch^*(f)\ge 0$ for all $v$ and all $f$.
Each 4-face $f$ loses no charge, so
$\ch^*(f)=4-4=0$. By (R1), each $6^+$-face $f$ has
$\ch^*(f)=(d(f)-4)-d(f)\left(\frac{d(f)-4}{d(f)}\right)=0$. Note that the only
vertices that receive charge from their neighbors are 3-vertices that are not in
$T$, by (R2). Therefore, each vertex $v$ of level 3 has $d(v)\geq4$ and at least
four neighbors that each receives no charge from $v$. So, (R2) implies $\ch^*(v)\geq
d(v)-4-\frac{1}{3}(d(v)-4)=\frac{2}{3}(d(v)-4)\geq0$. Moreover, each vertex $v$
of level 2 has $d(v)\geq4$ and loses no charge, so
$\ch^*(v)=\ch(v)=d(v)-4\geq0$. Finally, every 3-vertex $v$ has $\ch(v)=3-4=-1$,
so needs to receive at least 1. If $v$ is incident to at most 2 faces, then $v$ appears at least twice on the
boundary walk of an $8^+$-face, so $v$ receives at least $2(\frac{4}8)$ by (R1);
thus, $\ch^*(v)\ge 0$.  So assume that $v$ appears on 3 distinct faces. Let $M_v$ be the set of neighbors of $v$ of
level at least 3. Note, for every neighbor $x$ of $v$, that either (i) $x\in
M_v$ or (ii) the face 
opposite $x$ is a $6^+$-face; otherwise, $v$ is a good vertex with a good neighbor $x$, a
contradiction. Now $\ch^*(v)\geq -1+\frac{1}{3}|M_v|+\frac{1}{3}(3-|M_v|)=0$ by
(R1) and (R2). Thus, the lemma holds when $G$ is Type 1.

Now we assume that $G$ is Type 2. By the arguments above, $\ch^*(f)\ge 0$
for all $f$ and $\ch^*(v)\ge 0$ for every vertex $v$ of level 1, 2, or 3, i.e.,
for every $v\notin T$. It remains to show that $\sum_{v\in \Vz}\ch^*(v)>-8$. Let \Emph{$n_2$} be the
number of 2-vertices on $f_0$. Note that $n_2\leq6$; otherwise, since $G$ is
connected, $\Vz=V(G)$, a contradiction.

\begin{claim}
\label{deg2-pair}
If $w_1,\dots,w_m$ are 2-vertices on $f_0$ with $m\geq2$ and
$w_iw_{i+1}\in E(G)$ for all $i\in[m-1]$, then their incident face $f$ (other
than $f_0$) is a
$6^+$-face. Thus, each $w_i$, as well as the $3^+$-neighbors of $w_1$ and
$w_m$ on $f_0$, receives $\frac{d(f)-4}{d(f)}$ from $f$.
\end{claim}

\begin{proof}[Proof of claim]
The second statement follows directly from the first via (R1). To see the first
statement, let $v_1$ and $v_2$ be the neighbors of
$w_{\lfloor{\frac{m}{2}\rfloor}}$ and $w_{\lfloor{\frac{m}{2}\rfloor}+1}$ on
$f_0$ (other than each other), respectively, and assume $f$ is a 4-face. This
means $f=G[\{v_1,w_1,w_2,v_2\}]$ and $m=2$ since $d(v_1),d(v_2)\geq3$. Denote by
$v_3,v_4,v_5$ the remaining vertices on $f_0$. Now $G[\{v_1,v_2,v_3,v_4,v_5\}]$ is
a 5-cycle, a contradiction.
\end{proof}

%
Recall that $\ch^*(f)\ge 0$ for all $f$ and $\ch^*(v)\ge 0$ for all $v\notin
\Vz$.  So, to reach a contradiction to~\ref{charge-sum}, it suffices to show
that $\sum_{v\in \Vz}\ch^*(v)>-8$.
Note that each vertex on $f_0$ takes
$\frac{d(f_0)-4}{d(f_0)}=\frac{3}{7}$ from $f_0$ by (R1). Moreover, by (R2),
every $3^+$-vertex $v$ on $f_0$ satisfies $\ch^*(v)\geq
(d(v)-4)-\frac{1}{3}(d(v)-2)\geq \frac{-4}{3}$. Finally, every 2-vertex $v$
satisfies $\ch^*(v)\geq 2-4=-2$. Let \Emph{$g(\Vz)$} be the charge gained via
(R1), from faces other than $f_0$, in total by all vertices on $f_0$. Now
$\sum_{v\in \Vz}\ch^*(v)\geq
-2n_2-\frac{4}{3}(7-n_2)+7(\frac{3}{7})+g(\Vz)=\frac{-2}{3}n_2-\frac{19}{3}+g(\Vz)$. 
Let $h(\Vz):=\frac{2}{3}n_2-\frac{5}{3}$.\aside{$h(\Vz)$} Now, it suffices to show that
$g(\Vz)>h(\Vz)$ since this implies $\sum_{v\in \Vz}\ch^*(v)>-8$. 

\begin{figure}[!h]
\begin{subfigure}{0.33\textwidth}
\centering
\begin{tikzpicture}[scale=1.5, every node/.style={scale=1}, rotate=360/28, xscale=-1]
\foreach \an [count=\i] in {0,360/7,720/7,1080/7,1440/7,1800/7,2160/7}
{\tkzDefPoint(\an:1){v_\i}}
\tkzDrawPolygon[thick](v_1,v_...,v_7)
\tkzDrawPoints[scale=1.5, fill=white](v_1,v_...,v_7)
\tkzLabelPoints[below](v_6,v_7)
\tkzLabelPoints[above](v_3)
\tkzLabelPoints[right](v_4,v_5)
\tkzLabelPoints[left](v_1,v_2)
\tkzDefPointBy[reflection = over v_2--v_7](v_1) \tkzGetPoint{w_1} 
\tkzDefPointBy[reflection = over v_2--v_4](v_3) \tkzGetPoint{w_3}
\tkzDefPointBy[reflection = over v_4--v_6](v_5) \tkzGetPoint{w_5}
\tkzDrawPoints[scale=1.5, fill=white](w_1,w_3,w_5)
\tkzLabelPoints[above](w_3)
\tkzLabelPoints[right](w_5)
\tkzLabelPoints[left](w_1)
\draw[thick] (v_2) -- (w_1) -- (v_7) (v_2) -- (w_3) -- (v_4) (v_4) -- (w_5) -- (v_6);
\tkzDrawPoints[scale=1.5, fill=white](w_1,w_3,w_5)
\tkzDrawPoints[scale=1.5, fill=white](v_1,v_...,v_7)
\end{tikzpicture}
\end{subfigure}%
\begin{subfigure}{0.33\textwidth}
\centering
\begin{tikzpicture}[scale=1.5, every node/.style={scale=1}, rotate=360/28, xscale=-1]
\foreach \an [count=\i] in {0,360/7,720/7,1080/7,1440/7,1800/7,2160/7}
{\tkzDefPoint(\an:1){v_\i}}
\tkzDrawPolygon[thick](v_1,v_...,v_7)
\tkzDrawPoints[scale=1.5, fill=white](v_1,v_...,v_7)
\tkzLabelPoints[below](v_6,v_7)
\tkzLabelPoints[above](v_3)
\tkzLabelPoints[right](v_4,v_5)
\tkzLabelPoints[left](v_1,v_2)
\tkzDefPointBy[reflection = over v_1--v_3](v_2) \tkzGetPoint{w_1} 
\tkzDefPointBy[reflection = over v_4--v_6](v_5) \tkzGetPoint{w_5}
\tkzDrawPoints[scale=1.5, fill=white](w_1,w_5)
\tkzLabelPoints[right](w_5)
\tkzLabelPoints[left](w_1)
\draw[thick] (v_2) -- (w_1) -- (v_7) (w_1) -- (v_4) (v_4) -- (w_5) -- (v_6);
\tkzDrawPoints[scale=1.5, fill=white](v_1,v_...,v_7)
\tkzDrawPoints[scale=1.5, fill=white](w_1,w_5)
\end{tikzpicture}
\end{subfigure}%
\begin{subfigure}{0.33\textwidth}
\centering
\begin{tikzpicture}[scale=1.5, every node/.style={scale=1}, rotate=360/28, xscale=-1]
\foreach \an [count=\i] in {0,360/7,720/7,1080/7,1440/7,1800/7,2160/7}
{\tkzDefPoint(\an:1){v_\i}}
\tkzDrawPolygon[thick](v_1,v_...,v_7)
\tkzDrawPoints[scale=1.5, fill=white](v_1,v_...,v_7)
\tkzLabelPoints[below](v_6,v_7)
\tkzLabelPoints[above](v_3)
\tkzLabelPoints[right](v_4,v_5)
\tkzLabelPoints[left](v_1,v_2)
\tkzDefPointBy[reflection = over v_2--v_4](v_3) \tkzGetPoint{w_3}
\tkzDefShiftPoint[w_3](-80:0.5){w_1}
\tkzDrawPoints[scale=1.5, fill=white](w_1,w_3)
\tkzLabelPoints[above](w_3)
\tkzLabelPoints[left](w_1)
\draw[thick] (v_2) -- (w_1) -- (v_7) (v_2) -- (w_3) -- (v_4) (v_4) -- (w_1) -- (v_6);
\tkzDrawPoints[scale=1.5, fill=white](v_1,v_...,v_7)
\tkzDrawPoints[scale=1.5, fill=white](w_1,w_3)
\end{tikzpicture}
\end{subfigure}
\caption{The three instances of Case 2, in the proof of the
Key Lemma
when no 2-vertices on $f_0$ are adjacent. Left:
The nonneighbors of $v_1$, $v_3$, and $v_5$ on their incident 4-faces are
distinct. Center: The nonneighbors of $v_1$ and $v_3$ on their incident 4-faces
are the same, i.e., $w_1=w_3$. Right: The nonneighbors of $v_1$ and $v_5$ on
their incident 4-faces are the same, i.e., $w_1=w_5$.}
\label{case2}
\end{figure}
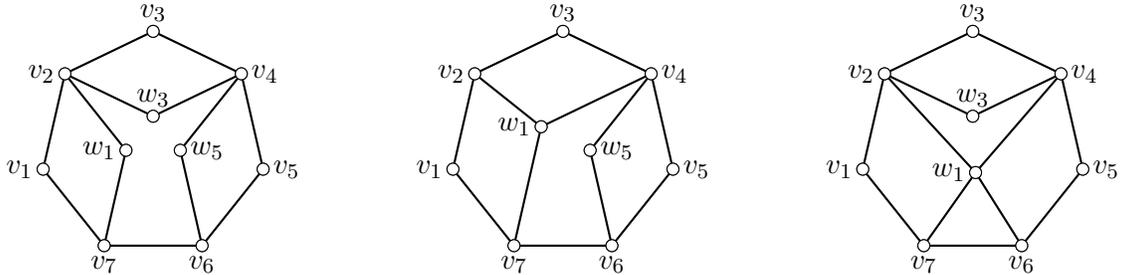

\textbf{Case 1: $\bm{n_2\leq2}$.} Now $h(\Vz)\leq\frac{-1}{3}$, so $g(\Vz)\ge
0>h(\Vz)$.

\textbf{Case 2: $\bm{n_2=3}$.} If $f_0$ has at least two adjacent 2-vertices,
then $g(\Vz)\geq\frac{1}{3}(4)>\frac{1}{3}=h(\Vz)$, by
Claim~\ref{deg2-pair}. So, we assume no pair of 2-vertices are adjacent on
$f_0$. Let $v_1,\dots,v_7$ be the vertices on $f_0$ in cyclic order. By
symmetry, assume $v_1$, $v_3$, and $v_5$ are 2-vertices. Note that at least one
2-vertex is incident with a $6^+$-face. To see this, assume every 2-vertex is
incident with a 4-face. Let $w_1$, $w_3$, and $w_5$ be the inner vertices incident
with the 4-faces of $v_1$, $v_3$, and $v_5$, respectively; see
Figure~\ref{case2}.  (Note that the boundary cycle of $f_0$ cannot have a chord,
since any chord would create either a 3-cycle or a 5-cycle, both of which are
forbidden by hypothesis.)
Observe that $w_1$, $w_3$, and $w_5$ are distinct.
Otherwise, by symmetry, either $w_1=w_3$ or $w_1=w_5$.  In the former case
$G[\{w_1,v_4,v_5,v_6,v_7\}]$ is a 5-cycle, and in the latter $G[\{w_1,v_6,v_7\}]$
is a 3-cycle, both of which are contradictions. But now
$G[\{w_1,v_2,w_3,v_4,v_5,v_6,v_7\}]$ is a 7-cycle that separates $v_3$ from $w_5$, a
contradiction.
Hence, by symmetry, $v_1$ is incident with a $6^+$-face $f$, 
so $v_1$, $v_2$, and $v_7$ each get at least $\frac{1}{3}$ from $f$ by (R1).
Thus, $g(\Vz)\geq\frac{1}{3}(3)>\frac{1}{3}=h(\Vz)$.

\textbf{Case 3: $\bm{n_2=4}$.} By Pigeonhole, at least two 2-vertices are
adjacent on $f_0$. So, $g(\Vz)\geq\frac{1}{3}(4)>1=h(\Vz)$ by Claim~\ref{deg2-pair}.

\textbf{Case 4: $\bm{n_2=5}$.} By Pigeonhole, either (a) at least four
2-vertices on $f_0$ induce a path, or (b) three 2-vertices on $f_0$ induce a
path $P_1$ and the other two 2-vertices on $f_0$ induce another path $P_2$, and
no vertex on $P_1$ is adjacent to a vertex on $P_2$. In both cases, at least 6
vertices on $f_0$ get at least $\frac{1}{3}$ via (R1) from faces other than
$f_0$, by Claim~\ref{deg2-pair}. So, $g(\Vz)\geq\frac{1}{3}(6)>\frac{5}{3}=h(\Vz)$.

\textbf{Case 5: $\bm{n_2=6}$.} Now the face $f$ incident to all 2-vertices
(other than $f_0$) is a $7^+$-face. So, each of the six 2-vertices (as well as
the single $3^+$-vertex) on $f_0$ gets at least $\frac{3}{7}$ by
Claim~\ref{deg2-pair}.
Thus, $g(\Vz)\geq\frac{3}{7}(7)>\frac{7}{3}=h(\Vz)$. 
\end{proof}


\begin{cor}
\label{goodvertex}
If $G$ is a plane graph with no 3-cycles, with no 5-cycles, and with
$\delta(G)\ge 3$, then $G$
contains a good vertex $v$. Furthermore, identifying $v$ and its nonneighbor on
an incident 4-face results in a smaller plane graph with no 3-cycles and no 5-cycles.
\end{cor}

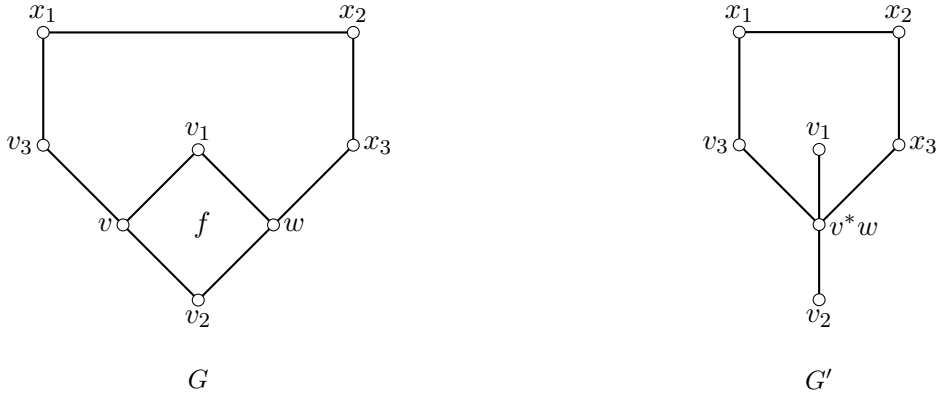
\begin{figure}[!h]
\begin{subfigure}{0.5\textwidth}
\centering
\begin{tikzpicture}[xscale=-1]
\tkzDefPoint(0:1){v}
\tkzDefPoint(90:1){v_1}
\tkzDefPoint(180:1){w}
\tkzDefPoint(270:1){v_2}
\tkzDefShiftPoint[v](45:1.5){v_3}
\tkzDefShiftPoint[v_3](90:1.5){x_1}
\tkzDefShiftPoint[w](135:1.5){x_3}
\tkzDefShiftPoint[x_3](90:1.5){x_2}
\tkzDrawPolygon[thick](v,v_1,w,v_2)
\draw[thick] (v) -- (v_3) -- (x_1) -- (x_2) -- (x_3) -- (w);
\tkzDrawPoints[scale=1.5, fill=white](v,v_1,v_2,v_3,w,x_1,x_2,x_3)
\tkzLabelPoints[below](v_2)
\tkzLabelPoints[left](v,v_3)
\tkzLabelPoints[above](v_1,x_1,x_2)
\tkzLabelPoints[right](w,x_3)
\tkzDefPoint(180:-.2){f}
\tkzLabelPoints[right](f)
\tkzDefShiftPoint[v](-.7,-1.75){G}
\draw (v) ++ (-1,-2.05) node[draw=none] {\small{$G$}};
\end{tikzpicture}
\end{subfigure}%
\begin{subfigure}{0.5\textwidth}
\centering
\begin{tikzpicture}[xscale=-1]
\tkzDefPoint(0:1){\vw}
\tkzDefShiftPoint[\vw](90:1){v_1}
\tkzDefShiftPoint[\vw](-90:1){v_2}
\tkzDefShiftPoint[\vw](45:1.5){v_3}
\tkzDefShiftPoint[v_3](90:1.5){x_1}
\tkzDefShiftPoint[\vw](135:1.5){x_3}
\tkzDefShiftPoint[x_3](90:1.5){x_2}
\draw[thick] (v_1) -- (\vw) -- (v_2) (\vw) -- (v_3) -- (x_1) -- (x_2) -- (x_3) -- cycle;
\tkzDrawPoints[scale=1.5, fill=white](\vw,v_1,v_2,v_3,x_1,x_2,x_3)
\tkzLabelPoints[below](v_2)
\tkzLabelPoints[left](v_3)
\tkzLabelPoints[above](v_1,x_1,x_2)
\tkzLabelPoints[right](\vw,x_3)
\draw (\vw) ++ (0,-2.05) node[draw=none] {\small{$G'$}};
\end{tikzpicture}
\end{subfigure}
\caption{Now $vv_3x_1x_2x_3wv_1$ is a separating 7-cycle in $G$ if
$(\vw)v_3x_1x_2x_3$ is a 5-cycle in $G'$.}
\label{corollary-fig}
\end{figure}
\begin{proof}
Assume the corollary is false and $G$ is a smallest counterexample.
Now $G$ is connected; otherwise, the result holds, by the minimality of $G$, for
each component of $G$. Suppose first that $G$ has no separating 7-cycles. By
the \hyperlink{key-lemma}{Key Lemma} for Type 1 graphs\footnote{Although $G$ may
also satisfy the criteria to be a Type 2 graph, by hypothesis it satisfies the
criteria to be a Type~1 graph, so we invoke the \hyperlink{key-lemma}{Key Lemma}
for Type 1 graphs.}, $G$ contains a good vertex $v$. Let
$v_1,v_2,v_3$ be the neighbors of $v$ with $v_3$ being a good neighbor of $v$
and $f$ being the 4-face opposite of $v_3$. Let $w$ be the nonneighbor of $v$
on $f$; see Figure~\ref{corollary-fig}. Form $G'$ from $G$ by identifying $v$ with $w$ to
form a new vertex 
$\vw$. Since $v$ and $w$ lie on a 4-face, $G'$ is a plane graph. If $G'$ has a
3-cycle, then $G$ has a 5-cycle, a contradiction. 
Suppose $G'$ has a 5-cycle, $(\vw)v_3x_1x_2x_3$; call it $C'$.
Since $d_G(v_1)\ge 3$, we note that $v_1$ has a neighbor $y$ in $G$ (other than
$v$ and $w$) and in $G'$ we know $y$ lies in the interior of $C'$.
Thus, in $G$ the 
7-cycle $vv_3x_1x_2x_3wv_1$ separates $y$ from $v_2$, again a
contradiction.
Hence, $G'$ has no 3-cycles and no 5-cycles. So, the corollary holds when $G$
has no separating 7-cycles.

\def\Cin{C_{\textrm{in}}}
Suppose instead that $G$ has a separating $7$-cycle and let $C$ be an innermost
such cycle. Denote by \Emph{$\Cin$} the subgraph induced by vertices that
lie on $C$ and inside $C$. 
Observe that $\Cin$ is a plane graph with no $3$-cycles, no $5$-cycles, and no
separating $7$-cycles. Further, $\delta(\Cin) \geq 2$ with every $2$-vertex of $\Cin$ lying on the
outer face of $\Cin$, which is induced by $V(C)$. 
By the \hyperlink{key-lemma}{Key Lemma}, 
$\Cin$ contains a good
vertex $v$. Define $v_1$, $v_2$, $v_3$, and $w$ as in the previous paragraph.
Note that $v,v_3\notin V(C)$, 
by the \hyperlink{key-lemma}{Key Lemma}. 
So, $d_G(v)=3$ and
$v_3$ has level at most 2 in $G$, i.e., $v$ is a good vertex in $G$. 
We note that $v_1$ and $v_2$ cannot both lie on $C$.  Suppose the contrary.
Since $C$ has length 7, some $v_1,v_2$-path $P$ in $C$ has odd length at most 5.
If $P$ has length 1 or 3, then $Pv_1vv_2$ is a 3-cycle or 5-cycle, a
contradiction.  So assume $P$ has length 5.  Now $Pv_1vv_2$ is a 7-cycle in
$\Cin$ that separates $v_3$ from $w$, a contradiction.  So assume by symmetry
that $v_1$ does not lie on $C$.

Form $G'$ from $G$ by identifying $v$ with $w$ to form a new vertex 
$\vw$. Clearly $G'$ is a plane graph with no $3$-cycles. Assume $G'$ has a
$5$-cycle. So $G$ has a $(v, w)$-path $P$ of length five. 
It is easy to check that $v_{1}$ is not an interior vertex of $P$. 
Let $Q:=Pvv_1w$, and denote its vertices by $v$, $v_3$, $x_1$, $x_2$, $x_3$,
$x_4$, $v_1$ (in order), where $x_4=w$.  See Figure~\ref{corollary-fig2}.
Since $d(v_1)\ge 3$, let $y$ be a neighbor of $v_1$
other than $v$ and $w$.  Note that $y$ is not on $Q$, since this would give a
7-cycle with a chord, and thus a 3-cycle or a 5-cycle in $G$, a contradiction.
So $Q$ separates $v_2$ from $y$.  Since $\Cin$ has no separating cycles, $Q$
must not lie entirely in $\Cin$.  Since $v$, $v_1$, and $v_3$ are inner vertices
of $\Cin$, the vertices of $Q$ outside $C$ are contained in $\{x_2,x_3\}$.

\begin{figure}[!h]
\begin{subfigure}{0.5\textwidth}
\centering
\begin{tikzpicture}[xscale=-1, bezier bounding box]
\tkzDefPoint(0:1){v}
\tkzDefPoint(90:1){v_1}
\tkzDefPoint(180:1){x_4}
\tkzDefPoint(270:1){v_2}
\tkzDefShiftPoint[v](45:1.5){v_3}
\tkzDefShiftPoint[v_3](90:1.5){x_1}
\tkzDefShiftPoint[x_4](135:1.5){x_3}
\tkzDefShiftPoint[x_3](90:1.5){x_2}
\tkzDefMidPoint(x_1,x_3) \tkzGetPoint{z} 
\tkzDrawPolygon[thick](v,v_1,x_4,v_2)
\draw[thick] (v) -- (v_3) -- (x_1) -- (x_2) -- (x_3) -- (x_4) (x_1) -- (z) -- (x_3);
\draw[thick] (x_1) edge[bend left=120, looseness=3.5, decorate, decoration={snake, amplitude=0.7mm}] (x_3);
\tkzDrawPoints[scale=1.5, fill=white](v,v_1,v_2,v_3,x_4,x_1,x_2,x_3,z)
\tkzLabelPoints[below](v_2)
\tkzLabelPoints[left](v,v_3)
\tkzLabelPoints[above](v_1,z,x_1,x_2)
\tkzLabelPoints[right](x_4,x_3)
\tkzDefPoint(180:-.2){f}
\tkzLabelPoints[right](f)
\tkzDefShiftPoint[v](315:2.2){C} \tkzLabelPoints[below](C)
\end{tikzpicture}
\end{subfigure}%
\begin{subfigure}{0.5\textwidth}
\centering
\begin{tikzpicture}[xscale=-1, bezier bounding box]
\tkzDefPoint(0:1){v}
\tkzDefPoint(90:1){v_1}
\tkzDefPoint(180:1){x_4}
\tkzDefPoint(270:1){v_2}
\tkzDefShiftPoint[v](45:1.5){v_3}
\tkzDefShiftPoint[v_3](90:1.5){x_1}
\tkzDefShiftPoint[x_4](135:1.5){x_3}
\tkzDefShiftPoint[x_3](90:1.5){x_2}
\tkzDefMidPoint(x_1,x_3) \tkzGetPoint{q} 
\tkzDefShiftPoint[q](210:1){p}
\tkzDrawPolygon[thick](v,v_1,x_4,v_2)
\draw[thick] (v) -- (v_3) -- (x_1) -- (x_2) -- (x_3) -- (x_4) (x_1) -- (q) -- (p) -- (x_4);
\draw[thick] (x_1) edge[bend left=130, looseness=3.5, decorate, decoration={snake, amplitude=0.7mm}] (x_4);
\tkzDrawPoints[scale=1.5, fill=white](v,v_1,v_2,v_3,x_4,x_1,x_2,x_3,p,q)
\tkzLabelPoints[below](v_2)
\tkzLabelPoints[left](v,v_3)
\tkzLabelPoints[above](q,v_1,x_1,x_2)
\tkzLabelPoints[right](x_4,x_3,p)
\tkzDefPoint(180:-.2){f}
\tkzLabelPoints[right](f)
\tkzDefShiftPoint[v](315:2.1){C} \tkzLabelPoints[below](C)
\end{tikzpicture}
\end{subfigure}
\caption{Two instances where the 7-cycle $Q$ does not lie entirely in $C$. Left: Here $x_2$ lies outside $C$ and $z$ is not part of $Q$. Right: Here $x_2$ and $x_3$ both lie outside $C$, and both $p$ and $q$ are not part of $Q$.}
\label{corollary-fig2}
\end{figure}
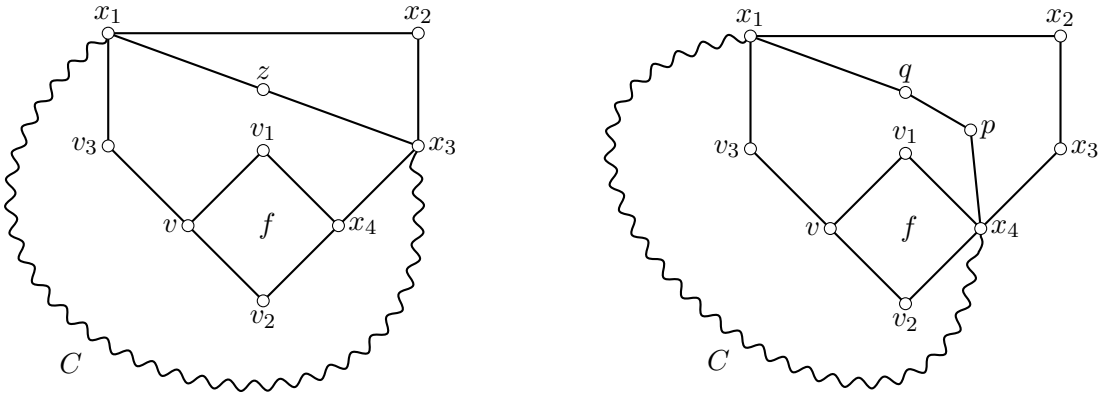

Suppose that $Q$ has exactly one vertex $x_i$ outside $C$. The
path $x_{i-1}x_ix_{i+1}$ together with the shorter path in $C$ between
$x_{i-1}$ and $x_{i+1}$ forms a cycle of length at most $5$. Thus, this must be
a $4$-cycle $x_{i-1}x_{i}x_{i+1}z$. If $z$ is not on $Q$, then $(Q\setminus
\{x_{i}\}) \cup x_{i-1}zx_{i+1}$ is a $7$-cycle of $\Cin$, which must be facial,
since $\Cin$ has no separating 7-cycles; see the left of  Figure~\ref{corollary-fig2}.  But now the path $vv_1w$ lies on both a 7-face and
4-face in $\Cin$, so $v_1$ is an inner 2-vertex of $\Cin$, a contradiction.
So instead $z$ must lie on $Q$.  But now either $x_{i-1}z$ or $x_{i+1}z$ must be
a chord of $Q$, a contradiction.

So assume instead that $Q$ has exactly two vertices outside $C$, and these are
$x_2$ and $x_3$.  Similar to above, the path $x_1x_2x_3x_4$ together with the
shorter path on $C$ between $x_1$ and $x_4$ forms a 4-cycle or $6$-cycle.
If this cycle has length $4$, then $x_1x_4$ is a chord of $Q$, a
contradiction. So the cycle must have length 6; we denote its vertices (in
order) by $x_1x_2x_3x_4pq$.  Note that neither $p$ nor $q$ is on $Q$,
since $p$ and $q$ lie on $C$, but all vertices of $Q$ other than $x_1, x_2, x_3, x_4$ are inner vertices of $\Cin$.
Thus, $(Q
\setminus \{x_2, x_3\}) \cup x_4pqx_1$ is a $7$-cycle in $\Cin$ as on the right in Figure~\ref{corollary-fig2}; again it must
be a facial 7-cycle.  But now the path $vv_1w$ lies on both a 7-face and a
4-face, so $v_1$ is an inner 2-vertex of $\Cin$, a contradiction.
\end{proof}

The proof of our \hyperlink{main-thm}{Main Theorem} (restated below) is now similar to a proof of
Bousquet and Heinrich \cite{BousquetHeinrich} who showed the same conclusion
for the smaller class of all bipartite planar graphs. 

\begin{mainthm}
If $G$ is a plane graph with no 3-cycles and no 5-cycles, then
$\diam(\mathcal{C}_5(G))\le 4n^2$. 
\end{mainthm}

\begin{proof}
Let $G$ be a plane graph with no 3-cycles and no 5-cycles. Let $\vph_A$ and
$\vph_B$ be two 5-colorings of $G$.  (Recall that $n:=|V(G)|$.)  We will show
that we can transform $\vph_A$ into $\vph_B$ by recoloring each vertex 
at most $4n$ times. We use induction on $n$.

\textbf{Case 1: $\bm{G}$ contains a vertex $\bm{v}$ such that
$\bm{d(v)\leq2}$.} 
Assume $n>1$; otherwise, the result follows trivially.
Let $\vph'_A$ and $\vph'_B$ be the restrictions of $\vph_A$ and $\vph_B$
to $G-v$. By induction, there exists a sequence $\mathcal{S}'$ of recolorings
that transforms $\vph'_A$ into $\vph'_B$ such that each vertex is recolored at
most $4(n-1)$ times. We extend $\mathcal{S}'$ to a sequence $\mathcal{S}$ of
recolorings in $G$.  To form $\mathcal{S}$ in $G$, we can perform each
recoloring step from $\mathcal{S}'$,
except when a neighbor $w$ of $v$ is to be recolored with the current
color $\alpha$ of $v$. In that case, we need to recolor $v$ before recoloring
its neighbor $w$. 
The number of colors unused on $N(v)\cup\{v\}$ is at least $5-(1+d(v))\ge 2$.
We recolor $v$ with one of these colors that is not the target color in the
next recoloring of a neighbor of $v$.  So, we only need to recolor $v$ at most once
for every two successive recolorings of neighbors of $v$. Finally, we may need
to recolor $v$ to its target color in $\vph_B$. Since $d(v)\leq2$ and each
neighbor of $v$ is recolored at most $4(n-1)$ times in $\mathcal{S}'$, the total
number of times we recolor $v$ is at most $\frac{2(4(n-1))}{2}+1<4n$.

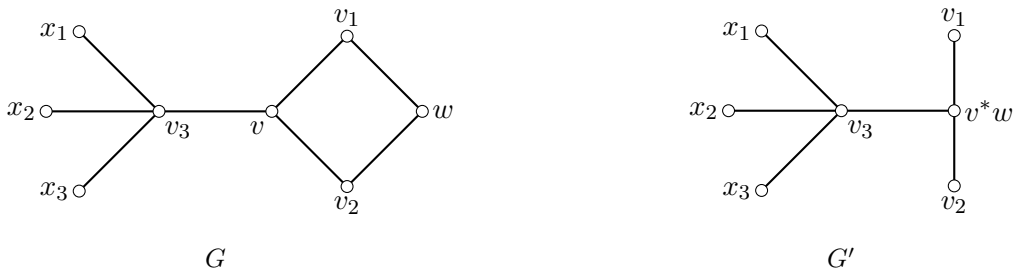
\begin{figure}[!b]
\begin{subfigure}{0.5\textwidth}
\centering
\begin{tikzpicture}[xscale=-1]
\tkzDefPoint(0:1){v}
\tkzDefPoint(90:1){v_1}
\tkzDefPoint(180:1){w}
\tkzDefPoint(270:1){v_2}
\tkzDefShiftPoint[v](0:1.5){v_3}
\tkzDefShiftPoint[v_3](0:1.5){x_2}
\tkzDefShiftPoint[v_3](45:1.5){x_1}
\tkzDefShiftPoint[v_3](-45:1.5){x_3}
\tkzDrawPolygon[thick](v,v_1,w,v_2)
\draw[thick] (v) -- (v_3) -- (x_1) (x_2) -- (v_3) -- (x_3);
\tkzDrawPoints[scale=1.5, fill=white](v,v_1,v_2,v_3,w,x_1,x_2,x_3)
\tkzLabelPoints[below](v_2)
\tkzLabelPoint[below](v_3){$~~~~v_3$}
\tkzLabelPoint[below](v){$v~~~$}
\tkzLabelPoints[left](x_1,x_2,x_3)
\tkzLabelPoints[above](v_1)
\tkzLabelPoints[right](w)

\draw (v_3) ++ (-.75,-1.95) node[draw=none] {\small{$G$}};
\end{tikzpicture}
\end{subfigure}%
\begin{subfigure}{0.5\textwidth}
\centering
\begin{tikzpicture}[xscale=-1]
\tkzDefPoint(0:1){\vw}
\tkzDefShiftPoint[\vw](90:1){v_1}
\tkzDefShiftPoint[\vw](-90:1){v_2}
\tkzDefShiftPoint[\vw](0:1.5){v_3}
\tkzDefShiftPoint[v_3](0:1.5){x_2}
\tkzDefShiftPoint[v_3](45:1.5){x_1}
\tkzDefShiftPoint[v_3](-45:1.5){x_3}
\draw[thick] (v_1) -- (\vw) -- (v_2) (\vw) -- (v_3) -- (x_1) (x_2) -- (v_3) -- (x_3);
\tkzDrawPoints[scale=1.5, fill=white](\vw,v_1,v_2,v_3,x_1,x_2,x_3)
\tkzLabelPoints[below](v_2)
\tkzLabelPoints[left](x_1,x_2,x_3)
\tkzLabelPoints[above](v_1)
\tkzLabelPoints[right](\vw)
\tkzLabelPoint[below](v_3){$~~~~v_3$}
\draw (v_3) ++ (0,-1.95) node[draw=none] {\small{$G'$}};
\end{tikzpicture}
\end{subfigure}
\caption{An instance of Case 2 in the proof of the Main Theorem.}
\label{identify}
\end{figure}
\textbf{Case 2: 
$\bm{\delta(G)\geq3}$.} 
Assume $n>1$; otherwise, the result follows trivially.
By Corollary~\ref{goodvertex}, $G$ contains a good
vertex $v$.\aside{$v$, $w$} Moreover, if \EmphE{$v_1$, $v_2$, $v_3$}{4mm} 
are the neighbors of $v$, where $v_3$ is
a good neighbor, and $w$ is the nonneighbor of $v$ on the 4-face opposite
$v_3$, then $G'$, which is formed from $G$ by identifying $v$ with $w$ into a
vertex $\vw$, is a planar graph with no 3-cycles and no 5-cycles. 

Note that if $\vph_A(v)\neq\vph_A(w)$, then we can transform $\vph_A$ into a
coloring $\vph'_A$ such that $\vph'_A(v)=\vph'_A(w)$ and every vertex is
recolored by this transformation at most twice. To see this, let
$\alpha:=\vph_A(w)$. Note that $\vph_A(v_1)\neq\alpha$ and
$\vph_A(v_2)\neq\alpha$ since $v_1,v_2\in N(w)$. If $\vph_A(v_3)\neq\alpha$,
then we recolor $v$ with $\alpha$ and we are done. So, assume
$\vph_A(v_3)=\alpha$. Recall that $v_3$ has level at most 2, so $v_3$ has at
most 3
neighbors of degree greater than 3. 
Fix $x_1,x_2,x_3\in N(v_3)$ such that each $4^+$-neighbor of $v_3$ is among
$\{x_1,x_2,x_3\}$.  Now there exists a color
$\beta\notin\{\alpha,\vph_A(x_1),\vph_A(x_2),\vph_A(x_3)\}$; see
Figure~\ref{identify}. If need be, we first recolor every $\beta$-colored
3-neighbor of $v_3$ (of which there may be arbitrarily many);
this is possible since at least one color does not appear
on the closed neighborhood of each 3-vertex. Observe that $v$ might possibly
be recolored, but $w$ is not recolored since $\vph_A(w)=\alpha\neq\beta$. We
now recolor $v_3$ with $\beta$, recolor $v$ with $\alpha$, and are done (since
$w$ and $v$ now both have color $\alpha$, we can identify them, to form a new
vertex).

Similarly, if $\vph_B(v)\neq\vph_B(w)$, then we can transform
$\vph_B$ into a coloring $\vph'_B$ such that $\vph'_B(v)=\vph'_B(w)$. Note that
$\vph'_A$ and $\vph'_B$ are proper 5-colorings of $G'$. By induction, there
exists a sequence $\mathcal{S}'$ that transforms $\vph'_A$ into $\vph'_B$ in
$G'$ such that each vertex is recolored at most $4(n-1)$ times. It is easy to
see that $\mathcal{S}'$ extends to a sequence $\mathcal{S}$ in $G$ such that
each vertex is recolored at most $4(n-1)$ times (for each recoloring of $\vw$
in $G'$ we recolor both $v$ and $w$ in $G$). 

Recall that we recolor each vertex at most twice when forming $\vph'_A$ from
$\vph_A$ and at most twice when forming $\vph'_B$ from $\vph_B$.
Thus, the total number of times we recolor each vertex is at most $4(n-1)+4\leq 4n$.
\end{proof}

\section*{Acknowledgments}
Thanks to Tao Wang for alerting us to two inaccuracies in an earlier version and suggesting ways to fix them.


\bibliographystyle{plainurl}
\footnotesize{
\bibliography{references}

\begin{thebibliography}{10}

\bibitem{BBFHMP}
Valentin Bartier, Nicolas Bousquet, Carl Feghali, Marc Heinrich, Benjamin
  Moore, and Th\'{e}o Pierron.
\newblock Recolouring planar graphs of girth at least five.
\newblock 2021.
\newblock \href {http://arxiv.org/abs/2112.00631} {\path{arXiv:2112.00631}}.

\bibitem{BJLPP}
Marthe Bonamy, Matthew Johnson, Ioannis Lignos, Viresh Patel, and Dani\"{e}l
  Paulusma.
\newblock Reconfiguration graphs for vertex colourings of chordal and chordal
  bipartite graphs.
\newblock {\em J. Comb. Optim.}, 27(1):132--143, 2014.
\newblock \href {https://doi.org/10.1007/s10878-012-9490-y}
  {\path{doi:10.1007/s10878-012-9490-y}}.

\bibitem{BousquetHeinrich}
Nicolas Bousquet and Marc Heinrich.
\newblock A polynomial version of {C}ereceda's conjecture.
\newblock {\em J. Combin. Theory Ser. B}, 155:1--16, 2022.
\newblock \href {http://arxiv.org/abs/1903.05619} {\path{arXiv:1903.05619}}.

\bibitem{CvJ}
Luis Cereceda, Jan van~den Heuvel, and Matthew Johnson.
\newblock Connectedness of the graph of vertex-colourings.
\newblock {\em Discrete Math.}, 308(5-6):913--919, 2008.
\newblock \href {https://doi.org/10.1016/j.disc.2007.07.028}
  {\path{doi:10.1016/j.disc.2007.07.028}}.

\bibitem{CvJ2}
Luis Cereceda, Jan van~den Heuvel, and Matthew Johnson.
\newblock Mixing 3-colourings in bipartite graphs.
\newblock {\em European J. Combin.}, 30(7):1593--1606, 2009.
\newblock \href {https://doi.org/10.1016/j.ejc.2009.03.011}
  {\path{doi:10.1016/j.ejc.2009.03.011}}.

\bibitem{DKT}
Zden{\v{e}}k Dvo{\v{r}}{\'a}k, Daniel Kr{\'a}l, and Robin Thomas.
\newblock Coloring triangle-free graphs on surfaces. {E}xtended abstract.
\newblock In {\em Proceedings of the {T}wentieth {A}nnual {ACM}-{SIAM}
  {S}ymposium on {D}iscrete {A}lgorithms}, pages 120--129. SIAM, Philadelphia,
  PA, 2009.

\bibitem{DvorakFeghali}
Zden\v{e}k Dvo\v{r}\'{a}k and Carl Feghali.
\newblock A {T}homassen-type method for planar graph recoloring.
\newblock {\em European J. Combin.}, 95:Paper No. 103319, 12, 2021.
\newblock \href {http://arxiv.org/abs/2006.09269} {\path{arXiv:2006.09269}}.

\bibitem{DP}
Zden\v{e}k Dvo\v{r}\'{a}k and Luke Postle.
\newblock Correspondence coloring and its application to list-coloring planar
  graphs without cycles of lengths 4 to 8.
\newblock {\em J. Combin. Theory Ser. B}, 129:38--54, 2018.
\newblock \href {http://arxiv.org/abs/1508.03437} {\path{arXiv:1508.03437}}.

\bibitem{Carl}
Carl Feghali.
\newblock Reconfiguring colorings of graphs with bounded maximum average
  degree.
\newblock {\em J. Combin. Theory Ser. B}, 147:133--138, 2021.
\newblock \href {http://arxiv.org/abs/1904.12698} {\path{arXiv:1904.12698}}.

\bibitem{Nishimura}
Naomi Nishimura.
\newblock Introduction to reconfiguration.
\newblock {\em Algorithms (Basel)}, 11(4):Paper No. 52, 25, 2018.
\newblock \href {https://doi.org/10.3390/a11040052}
  {\path{doi:10.3390/a11040052}}.

\bibitem{vandenHeuvel}
Jan van~den Heuvel.
\newblock The complexity of change.
\newblock In {\em Surveys in combinatorics 2013}, volume 409 of {\em London
  Math. Soc. Lecture Note Ser.}, pages 127--160. Cambridge Univ. Press,
  Cambridge, 2013.
\newblock \href {http://arxiv.org/abs/1312.2816} {\path{arXiv:1312.2816}}.

\bibitem{West}
Douglas~B. West.
\newblock {\em Introduction to graph theory}.
\newblock Prentice Hall, Inc., Upper Saddle River, NJ, 1996.

\end{thebibliography}
}

\end{document}